\documentclass[a4paper,DIV=classic, 11pt]{myart}
\usepackage[normalem]{ulem}

\newcommand{\norm}[1]{\left\Vert#1\right\Vert}
\newcommand{\abs}[1]{\left\vert#1\right\vert}

\renewcommand{\mid}{\,|\,}

\usepackage{marginnote}
\usepackage{todonotes}

\begin{document}
\title{Multidimensional Markov FBSDEs with superquadratic growth}
\keyAMSClassification{(2000) 60H10, 60H07, 60H30}
\author[a,1]{Michael Kupper}
\author[b,2]{Peng Luo}
\author[c,3]{Ludovic Tangpi}
\address[a]{University of Konstanz, Universit\"atsstra\ss e.~10, D-78457 Konstanz, Germany}
\address[b]{ETH Z\"urich, R\"amistrasse 101, 8092 Z\"urich, Switzerland}
\address[c]{University of Vienna, Faculty of Mathematics, Oskar-Morgenstern-Platz 1, A-1090 Wien, Austria}
\eMail[1]{kupper@uni-konstanz.de}
\eMail[2]{peng.luo@math.ethz.ch}
\eMail[3]{ludovic.tangpi@univie.ac.at}

\abstract{

We give local and global existence and uniqueness results for systems of coupled FBSDEs in the multidimensional setting and with generators allowed to grow arbitrarily fast in the control variable.
Our results are based on Malliavin calculus arguments and pasting techniques.
}

\keyWords{Markovian FBSDE, superquadratic, existence and uniqueness}
\maketitle

\section{Introduction}

Given a multidimensional Brownian motion $W$ on a probability space,  consider the system of forward and backward stochastic differential equations
\begin{equation}
\label{eq:fbsde1}
	\begin{cases}
		X_t &=x+\int_{0}^{t}b_s(X_s,Y_s)\,ds+\int_{0}^{t}\sigma_sdW_s\\
	Y_t &=h(X_T)+\int_{t}^{T}g_s(X_s,Y_s,Z_s)\,ds-\int_{t}^{T}Z_sdW_s, \quad t\in [0,T]
	\end{cases}
\end{equation}
where $x$ is a known initial value, $T>0$ and $b$, $\sigma$, $g$ and $h$ are given functions.
In this paper, we give conditions under which the system admits a unique solution in the case where the value process $Y$ is multidimensional and the generator $g$ can grow arbitrarily fast in the control process $Z$.

The focus in this paper is on Markovian systems, in which the functions $b$, $\sigma$, $g$ and $h$ are deterministic.
We consider generators that are Lipschitz continuous in $X$ and $Y$ and locally Lipschitz continuous in $Z$.
In this setting, and for one-dimensional value processes, the decoupled system (with $b$ depending only on $X$) has been solved by \citet{Che-Nam} based on Malliavin calculus arguments.
In fact, using that for Lipschitz continuous generators the trace of the Malliavin derivative of the value process $Y$ is a version of the control process, they showed that the control process can be uniformly bounded, hence enabling solvability for locally Lipschitz generators by a truncation argument.
We make ample use of their method to deal with the coupled system \eqref{eq:fbsde1}.
In this case, we propose a Picard iteration scheme which yields a sequence that can be proved to be a Cauchy sequence in an appropriate Banach space under uniform boundedness of the control processes derived using Malliavin calculus arguments provided that the time horizon is small enough.
Moreover using the PDE representation of Markovian Lipschitz FBSDEs as developed for instance in \citet{Delarue} and a pasting precedure, we construct a unique global solution for generators with growth specified on the diagonal and under additional assumptions, mainly non-degeneracy of the volatility $\sigma$, see Theorem \ref{thm:fbsde-global}.

Systems such as \eqref{eq:fbsde1} naturally appear in numerous areas of applied mathematics including stochastic control and mathematical finance, see \citet{xyz}, \citet{karoui01} and \citet{Hor-etal}.
As shown for instance in \citet{Ma-Pro-Yong} and \citet{Che-Nam}, in the Markovian case, FBSDEs can be linked to parabolic PDEs.
More recently \citet{Fro-Imk-Pro} proved that FBSDEs can be used in the study of the Skorokhod embedding problem.

BSDEs and FBSDEs with Lipschitz continuous generators are well understood, we refer to \citet{karoui01} and \citet{Delarue}.
If $Y$ is one-dimensional and $g$ is allowed to have quadratic growth in the control process $Z$, BSDEs' solutions have been obtained by \citet{kobylanski01}, \citet{barrieu2013} and \citet{Bri-Hu1,Bri-Hu} under various assumptions on the terminal condition $\xi = h(X_T)$.
 We further refer to \citet{Delbaen11}, \citet{DHK1101}, \citet{Che-Nam} and \citet{optimierung} for results on one-dimensional BSDEs and FBSDEs with superquadratic growth.
Mainly due to the absence of comparison principle, general solvability of multidimensional BSDEs with quadratic growth is less well understood.
Some important progress have been achieved recently for BSDEs with small terminal conditions, see \citet{tevzadze}  and for more recent development, see \citet{Hu-Tang}, \citet{fbsdedq}, \citet{Jam-Kup-Luo}, \citet{Che-Nam_mbsde}, \citet{frei-split} and \citet{Xing-Zit}.

To the best of our knowledge, the only works studying well-posedness of coupled FBSDEs with quadratic growth are the article of \citet{Ant-Ham} and the preprints of \citet{fbsdedq} and \citet{Fro-Imk}.

In \cite{Ant-Ham} the authors consider a one-dimensional equation with one dimensional Brownian motion and impose monotonicity conditions on the coefficient so that comparison principles for SDEs and BSDEs can be applied.
A (non-necessarily unique) solution is then obtained by monotone convergence of an iterative scheme.
In \cite{Fro-Imk}, a fully coupled Markovian FBSDE is considered with multidimensional forward and value processes and locally Lipschitz generator in $(Y,Z)$ and a existence of a unique local solution is obtained using the technique of decoupling fields and an extension to global solutions is proposed.
Although the (non-Markovian) system studied in \citet{fbsdedq} is the same as the one considered in the present paper, the techniques here are essentially different.
Furthermore, the main results we present here can be extended to the non-Markovian setting and to random diffusion coefficient (when $\sigma$ depends on $X$ and $Y$) under stronger assumptions involving the Malliavin derivatives of $g$ and $h$. We refer to the Ph.D. thesis of \citet{Luo-thesis} for details.

In the next section, we present our probabilistic setting and the principal results of the paper.
In Section \ref{Sec:mBSDE} we prove local solvability of multidimensional BSDE with superquadratic growth and give conditions guaranteeing global solvability. Section \ref{sec:FBSDE} is dedicated to the proofs of the main results.

\section{Main results}

Let $(\Omega, {\cal F}, ({\cal F}_t)_{t \in [0,T]}, P)$
be a filtered probability space, where $({\cal F}_t)_{t \in [0,T]}$ is the augmented filtration generated
by a $d$-dimensional Brownian motion $W$, and ${\cal F} = {\cal F}_T$ for a finite time horizon $T \in (0, \infty)$.
The product  $\Omega \times [0,T]$ is endowed with the predictable $\sigma$-algebra. Subsets of $\mathbb{R}^k$ and
$\mathbb{R}^{k\times k}$, $k\in\mathbb{N}$, are always endowed with the Borel $\sigma$-algebra induced by the Euclidean norm $|\cdot|$. The interval $[0,T]$ is equipped with the Lebesgue measure. Unless otherwise stated, all equalities and inequalities between random variables and processes will be understood in the $P$-almost sure and $P\otimes dt$-almost sure sense, respectively.
For $p \in [1, \infty]$ and $k\in\mathbb{N}$, denote by ${\cal S}^p(\mathbb{R}^k)$
the space of all predictable continuous processes $X$
with values in $\mathbb{R}^k$ such that
$\norm{X}_{{\cal S}^p(\mathbb{R}^k)}^p :=  E[(\sup\nolimits_{t \in [0,T]}\abs{X_t})^p] < \infty$,
and by ${\cal H}^p(\mathbb{R}^{k})$ the space of all predictable processes $Z$ with values in $\mathbb{R}^{k}$ such that $\norm{Z}_{{\cal H}^p(\mathbb{R}^{k})}^p := E[(\int_0^T\abs{Z_u}^2\,du)^{p/2}] < \infty$.


Let $l, m \in \mathbb{N}$ be fixed.
The solution of \eqref{eq:fbsde1}  with values in $\mathbb{R}^m\times \mathbb{R}^l\times \mathbb{R}^{l\times d}$ can be obtained under the following conditions:
\begin{enumerate}[label=(\textsc{A1}),leftmargin=40pt]
	\item $b:  [0,T]\times \mathbb{R}^{l}  \to \mathbb{R}^m $ is a continuous function and there exist $k_1, k_2, \lambda_1\ge 0$ such that
	\begin{align*}
		\abs{b_t(x,y) - b_t(x',y')} \le k_1\abs{x - x'} + k_2\abs{y - y'} \text{ and }
		 \abs{b_t(x,y)} \le \lambda_1(1 + \abs{x} + \abs{y} )
	\end{align*}
	for all $x, x' \in \mathbb{R}^m$ and $y,y' \in \mathbb{R}^l$.\label{a1}
\end{enumerate}
\begin{enumerate}[label = (\textsc{A2}), leftmargin = 40pt]
	\item $\sigma: [0,T] \to \mathbb{R}^{m\times d}$ is a measurable function and there is $\lambda_2 \ge 0$ such that $\abs{\sigma_t} \le \lambda_2$ for all $t \in [0,T]$.\label{a2}
\end{enumerate}
\begin{enumerate}[label = (\textsc{A3}), leftmargin = 40pt]
	\item $h: \mathbb{R}^m \to \mathbb{R}^{l}$ is a continuous function and there exists $k_5 \ge 0$ such that
	\begin{align*}
		\abs{h(x) - h(x')} &\le k_5\abs{x - x'}
	\end{align*}
	for all $x, x' \in \mathbb{R}^m$. \label{a3}
\end{enumerate}
\begin{enumerate}[label = (\textsc{A4}), leftmargin = 40pt]
	\item $g: [0,T]\times \mathbb{R}^m \times \mathbb{R}^{l}\times \mathbb{R}^{l\times d}  \to \mathbb{R}^{l}$ is a continuous function satisfying  $g_\cdot(0,0,0) \in L^2([0,T])$, and there exist $k_3, k_4\ge 0$ and a nondecreasing function $\rho: \mathbb{R}_+ \to \mathbb{R}_+$ such that
	\begin{align}
	\label{eq:a3}
		\abs{g_t(x, y,z) - g_t(x',y,z)} &\le k_3\abs{x - x'}
	\end{align}
	for all $x, x' \in \mathbb{R}^m$, $y \in \mathbb{R}^{l}$ and $z \in \mathbb{R}^{l\times d}$ such that $\abs{z}\le M:=4\lambda_2k_5\sqrt{dl}$ and
	\begin{align*}
		\abs{g_t(x, y,z) - g_t(x,y',z')} &\le  k_4\abs{y - y'}+ \rho\left(\abs{z}\vee \abs{z'}\right)\abs{z - z'}
	\end{align*}
	for all $x \in \mathbb{R}^m$, $y, y' \in \mathbb{R}^{l}$ and $z,z' \in \mathbb{R}^{l\times d}$. \label{a4}
\end{enumerate}
\begin{enumerate}[label = (\textsc{A5}), leftmargin = 40pt]
	\item There exists a constant $K\ge 0$ such that
	\begin{equation*}
	 	\abs{g_t(x, y, z) - g_t(x', y, z) - g_t(x, y', z')+g_t(x', y', z')}\le K\abs{x - x'}\left(\abs{y-y'} + \abs{z -z'} \right)
	 \end{equation*}
	 for all $t \in [0,T]$, $x, x' \in \mathbb{R}^m$, $y,y' \in \mathbb{R}^{l}$ and $z, z' \in \mathbb{R}^{l\times d}$. \label{a5}
\end{enumerate}

 Our first main result ensures local existence and uniqueness for the coupled FBSDE \eqref{eq:fbsde1} under the previous assumptions.
 The proof is given in Section \ref{sec:FBSDE}.
\begin{theorem}
\label{thm:markov1}
	Assume that \ref{a1}-\ref{a5} hold. Then there exists a constant $C_{k,\lambda,l,d}>0$ depending only on $k_1$, $k_2$, $k_3$, $k_4$, $k_5$, $\lambda_2$, $l$ and $d$, such that the FBSDE \eqref{eq:fbsde1} has a unique solution $(X,Y,Z) \in {\cal S}^2(\mathbb{R}^m)\times{\cal S}^2(\mathbb{R}^l)\times {\cal S}^\infty(\mathbb{R}^{l\times d})$ with $|Z_t| \le M$,
    whenever $T\le C_{k,\lambda,l,d}$.
\end{theorem}
Local existence results as Theorem \ref{thm:markov1} above have been obtained in \cite[Theorem 3]{Fro-Imk} and \cite[Theorem 2.1]{fbsdedq} in essentially different settings and with different methods.
A natural question is to find conditions under which the above result can be extended to obtain global solvability.
\citet{Fro-Imk} propose a concatenation procedure allowing to prove existence of solutions to a fully coupled FBSDE on a "maximal interval".
In the present setting, under additional assumptions, a pasting method based on PDEs allows to get global existence and uniqueness for the FBSDE \eqref{eq:fbsde1}.
The proof is also given in Section \ref{sec:FBSDE}.
\begin{enumerate}[label = (\textsc{A4'}), leftmargin = 40pt]
	\item $g: [0,T]\times \mathbb{R}^m \times \mathbb{R}^{l}\times \mathbb{R}^{l\times d}  \to \mathbb{R}^{l}$ is a continuous function satisfying  $g_\cdot(0,0,0) \in L^2([0,T])$, and there exist $k_3, k_4\ge 0$ and a nondecreasing function $\rho: \mathbb{R}_+ \to \mathbb{R}_+$ such that
	\begin{align*}
		\abs{g_t^i(x, y,z) - g_t^i(x',y',z')} &\le k_3\abs{x - x'}+ k_4\abs{y - y'}+ \rho\left(\abs{z}\vee \abs{z'}\right)\abs{z^i - (z')^i}
	\end{align*}
	for all $i = 1, \dots, l$; $x, x' \in \mathbb{R}^m$, $y, y' \in \mathbb{R}^{l}$ and $z,z' \in \mathbb{R}^{l\times d}$. \label{a4prime}
\end{enumerate}
\begin{theorem}
\label{thm:fbsde-global}
	Assume that \ref{a1}, \ref{a2}, \ref{a3}, \ref{a4prime}, and \ref{a5} hold and there exist $\lambda_3> 0$; $\lambda_4,\lambda_5 \ge 0$
	such that\footnote{$\sigma_t^\ast$ is the transpose matrix of $\sigma_t$.}	
\begin{equation}
\label{eq:global-condi}
\begin{cases}
 \abs{b_t(x,y)}& \le \lambda_1(1 +  \abs{y})\\
	 \langle x, \sigma_t\sigma_t^\ast x\rangle& \ge \lambda_3|x|^2\\
	\abs{g_t(x,y,z)} &\le \lambda_4(1 + \abs{y} + \rho(\abs{z})\abs{z})\\
 \abs{h(x)} &\le \lambda_5
\end{cases}
\end{equation}
for all  $t \in [0,T]$, $x, x' \in \mathbb{R}^m$, $y, y' \in \mathbb{R}^{l}$ and $z,z' \in \mathbb{R}^{l\times d}$,
then the FBSDE \eqref{eq:fbsde1} has a unique global solution $(X, Y, Z) \in {\cal S}^2(\mathbb{R}^m)\times {\cal S}^2(\mathbb{R}^{l})\times {\cal S}^\infty(\mathbb{R}^{l\times d})$ such that $|Z_t|\leq \bar{M}$ for some constant $\bar{M} \ge 0$.
\end{theorem}
\begin{remark}
The following counterexample shows that the additional conditions in Theorem \ref{thm:fbsde-global} cannot be dropped without violating global solvability.
Consider the FBSDE
\begin{equation*}
	\begin{cases}
	X_t = \int_0^tY_u\,du\\
	Y_t =  \int_t^TkX_u\,du - \int_t^TZ_u\,dW_u.
	\end{cases}
\end{equation*}
This equation can be rewritten as
\begin{equation}
\label{eq:delay}
	Y_t = \int_t^T\int_0^skY_u\,du\,ds - \int_t^TZ_u\,dW_u.
\end{equation}
It has been shown in \cite[Example 3.2]{Del-Imk} that if $T\sqrt{k} < \frac{\pi}{2}$ then the BSDE with time-delayed generator \eqref{eq:delay} has a unique solution whereas if $T\sqrt{k} = \frac{\pi}{2}$, \eqref{eq:delay} may not have any solutions and if it does have one, there are infinitely many.
\end{remark}

Theorem \ref{thm:fbsde-global} extends to fully coupled generators provided that a more special structure is assumed.
The proof of the next proposition is given in Section \ref{sec:FBSDE1}.
\begin{proposition}
\label{thm:prop_coupled}
	If there exist an invertible matrix $\Gamma\in\mathbb{R}^{n\times n}$ and $\lambda_3> 0$; $\lambda_4,\lambda_5 \ge 0$ such that \ref{a1}, \ref{a2}, \ref{a3}, \ref{a4prime}, \ref{a5} and \eqref{eq:global-condi} hold for $\tilde{h}:=\Gamma h$, $\tilde{g}_t(x,y,z):=\Gamma g_t(x,\Gamma^{-1}y,\Gamma^{-1}z)$ and $\tilde{b}_t(x,y):=b_t(x,\Gamma^{-1}y)$, then the FBSDE \eqref{eq:fbsde1} has a unique global solution $(X, Y, Z) \in {\cal S}^2(\mathbb{R}^m)\times {\cal S}^2(\mathbb{R}^{2})\times {\cal S}^\infty(\mathbb{R}^{2\times d})$ such that $|Z_t|\leq \bar{M}$ for some constant $\bar{M} \ge 0$.
\end{proposition}
The point of the above proposition is that global solvability can be obtained for a generator that does not satisfy the "diagonally superquadratic" growth condition \ref{a4prime} provided that the transformed generator $\tilde{g}$ does.
Moreover, notice that Theorem \ref{thm:fbsde-global} and Proposition \ref{thm:prop_coupled} guarantee existence of a decoupling field, see \cite{Fro-Imk}. In particular, the boundedness of $Z$ guarantees a uniform Lipschitz property of the decoupling field.

Theorem \ref{thm:markov1} relies on an existence result for multidimensional BSDEs presented in \citet{nam} and revisited in the next section.
\section{Multidimensional BSDEs with bounded Malliavin derivatives}
\label{Sec:mBSDE}

Let us introduce the spaces of Malliavin differentiable random variables and stochastic processes
${\cal D}^{1,2}(\mathbb{R}^l)$ and ${\cal L}^{1,2}_a(\mathbb{R}^l)$. For a thorough treatment of the theory of Malliavin calculus we refer to \citet{Nualart2006}.
Let ${\cal M}$ be the class of smooth random variables $\xi=(\xi^1,\dots,\xi^l)$ of the form
\begin{equation*}
	\xi^i = \varphi^i\Big(\int_0^T h^{i1}_s\,dW_s, \dots, \int_0^T h^{in}_s\,dW_s \Big)
\end{equation*}
where $\varphi^i$ is in the space $C_p^\infty(\mathbb{R}^{n};\mathbb{R})$  of infinitely continuously differentiable functions whose partial derivatives have polynomial growth, $h^{i1}, \dots, h^{in} \in L^2([0,T]; \mathbb{R}^d)$ and $n\ge 1$.
For every $\xi$ in ${\cal M}$ let the operator $D = (D^1, \dots, D^d):{\cal M}\to L^2(\Omega\times [0,T];\mathbb{R}^d)$ be  given by
\begin{equation*}
	D_t\xi^i := \sum_{j=1}^n\frac{\partial \varphi^i}{\partial x_{j}}\Big(\int_0^T h^{i1}_s\,dW_s, \dots, \int_0^T h^{in}_s\,dW_s \Big)h^{ij}_t, \quad 0\le t\le T,\,\, 1\le i\le l,
\end{equation*}
and the norm $	\norm{\xi}_{1,2} := (E[\abs{\xi}^2 + \int_0^T\abs{D_t\xi}^2\,dt  ] )^{1/2}$.
As shown in \citet{Nualart2006}, the operator $D$ extends to the closure ${\cal D}^{1,2}(\mathbb{R}^l)$ of the set ${\cal M}$ with respect to the norm $\norm{\cdot}_{1,2}$.
A random variable $\xi$ is Malliavin differentiable if $\xi \in {\cal D}^{1,2}(\mathbb{R}^l)$ and we denote by $D_t\xi$ its Malliavin derivative.
Denote by ${\cal L}^{1,2}_a(\mathbb{R}^{l})$ the space of processes $Y \in {\cal H}^2(\mathbb{R}^{l})$ such that
$Y_t \in {\cal D}^{1,2}(\mathbb{R}^{l})$ for all $t \in [0,T]$, the process $DY_t$ admits a square integrable progressively measurable version and
\begin{equation*}
	\norm{Y}_{{\cal L}^{1,2}_a(\mathbb{R}^l)}^2 := \norm{Y}_{{\cal H}^2(\mathbb{R}^l)} + E\Big[\int_0^T\int_0^T\abs{D_r Y_t}^2\,dr\,dt \Big] < \infty.
\end{equation*}

We next consider a system of superquadratic BSDEs of the form
\begin{equation}\label{sup}
Y_t=\xi+\int_{t}^{T}g_u(Y_u,Z_u)du-\int_{t}^{T}Z_udW_u
\end{equation}
satisfying the following conditions:
\begin{enumerate}[label = (\textsc{B1}), leftmargin = 40pt]
	\item $g: \Omega \times [0,T]\times \mathbb{R}^{l}\times \mathbb{R}^{{l}\times d}  \to \mathbb{R}^{l}$ is a measurable function and there exist a constant $B\in\mathbb{R}_+$ and a nondecreasing function $\rho: \mathbb{R}_+ \to \mathbb{R}_+$ such that
	\begin{align*}
		\abs{g_t(y,z) - g_t(y',z')} &\le B\abs{y - y'}+ \rho\left(\abs{z}\vee \abs{z'}\right)\abs{z - z'}
	\end{align*}
	for all $t\in [0,T]$, $y, y' \in \mathbb{R}^{l}$ and $z,z' \in \mathbb{R}^{{l}\times d}$\label{b1}.
\end{enumerate}	
\begin{enumerate}[label = (\textsc{B2}), leftmargin = 40pt]
	\item $\xi \in \mathcal{D}^{1,2}(\mathbb{R}^{l})$ and there exist constants $A_{ij}\ge 0$ such that $|D^j_t\xi^i| \le A_{ij}$
 for all $i = 1,\dots, {l}$, $j=1,\ldots,d$ and $t \in [0,T]$.\label{b2}
\end{enumerate}
\begin{enumerate}[label = (\textsc{B3}), leftmargin = 40pt]
	\item $g_\cdot(0, 0) \in\mathcal{H}^4(\mathbb{R}^l)$ and there exist
   Borel-measurable functions $q_{ij}:[0,T]\rightarrow\mathbb{R}_{+}$ satisfying $\int_{0}^{T}q^2_{ij}(t)dt<\infty$ and for every pair $(y,z)\in\mathbb{R}^{l}\times\mathbb{R}^{{l}\times d}$ with
 \begin{equation*}
 |z|\leq Q:= \sqrt{2\sum_{j=1}^d \Bigg(\sum_{i=1}^{l}A_{ij}^2+\sum_{i=1}^{l}\int_0^{T} q^2_{ij}(t) dt\Bigg)}
	 \end{equation*}
  it holds
\begin{itemize}
\item $g_\cdot(y,z)\in \mathcal{L}^{1,2}_a(\mathbb{R}^{l})$ with $ |D^j_ug^i_t(y, z)|\le q_{ij}(t)$
for all $i = 1, \dots, {l}$, $j=1,\ldots,d$ and $u\in[0,T]$,
\item for almost all $u\in[0,T]$ one has
	\begin{equation*}
		\abs{D_ug_t(y, z) - D_ug_t(y', z')} \le K_u(t)\left( \abs{y - y'} + \abs{z - z'} \right)
	\end{equation*}
for all $t\in [0,T]$, $y, y' \in \mathbb{R}^{l}$ and $z,z' \in \mathbb{R}^{{l}\times d}$
for some $\mathbb{R}_+$-valued adapted process $(K_u(t))_{t\in [0,T]}$ satisfying $\int_0^T\norm{K_u}_{{\cal H}^4(\mathbb{R})}^4\,du < \infty$.
\end{itemize}\label{b3}
\end{enumerate}

The following is an extension of \citet[Theorem 2.2]{Che-Nam} to the multidimensional case.
It was proved in \citet{nam} under slightly different assumptions.
We give the proof for the sake of completeness.

\begin{theorem}\label{thm:sup}
	Assume that \ref{b1}-\ref{b3} hold and $T\leq \frac{\log (2)}{2B+\rho^2(Q)+1}$. Then the BSDE \eqref{sup} admits a unique solution in $\mathcal{S}^{4}(\mathbb{R}^{l})\times\mathcal{S}^{\infty}(\mathbb{R}^{{l}\times d})$ and $|Z_t|\le Q$.
\end{theorem}

Consider the following stronger versions of the conditions \ref{b1} and \ref{b3}:
\begin{enumerate}[label = (\textsc{B1'}), leftmargin = 40pt]
	\item $g$ is  continuously differentiable in $(y,z)$ and there exist constants $B\in\mathbb{R}_+$ and $\rho\in\mathbb{R}_+$ such that $\abs{\partial_{y}g_t(y,z)} \le B$ and $\abs{\partial_{z}g_t(y,z)} \le \rho$
	for all $t\in [0,T]$, $y, y' \in \mathbb{R}^{l}$ and $z,z' \in \mathbb{R}^{{l}\times d}$.\label{b1prime}
\end{enumerate}	
\begin{enumerate}[label = (\textsc{B3'}), leftmargin = 40pt]
	\item The condition \ref{b3} holds for all $(y,z)\in\mathbb{R}^{l}\times\mathbb{R}^{{l}\times d}$.\label{b3prime}
\end{enumerate}
\begin{lemma}\label{le:sup}
	If \ref{b1prime}, \ref{b2} and \ref{b3prime} hold, then the BSDE \eqref{sup} admits a unique solution $(Y,Z)\in\mathcal{S}^4(\mathbb{R}^{l})\times\mathcal{H}^4(\mathbb{R}^{{l}\times d})$ and
	\begin{equation}
	\label{eq:boundZlemma}
		|Z^{j}_t|^2\leq \Bigg(\sum_{i=1}^{l}A^2_{ij}+\sum_{i=1}^{l}\int_t^{T}q^2_{ij}(s)e^{-\left(2B+\rho^2+1\right)(T-s)}ds\Bigg)e^{\left(2B+\rho^2+1\right)(T-t)}\quad \mbox{for all }j = 1, \dots, d.
	\end{equation}
\end{lemma}
\begin{proof}
	By \citet[Lemma 2.5]{Che-Nam}, the condition \ref{b2} implies $E\left[|\xi|^p\right] <+\infty$ for all $p\in [1, \infty)$.
	It follows from \citet[Theorem 5.1 and Proposition 5.3]{karoui01} that the BSDE \eqref{sup} has a unique solution $(Y,Z)\in\mathcal{S}^4(\mathbb{R}^{l})\times\mathcal{H}^4(\mathbb{R}^{{l}\times d})$, which is Malliavin differentiable.
    Moreover for every $i=1,\ldots,{l}$ and $j=1,\ldots,d$, the process $(D^j_rY^i_t,D^j_rZ^i_t)_{t\in [0,T]}$ has a version $(U^{ij,r}_t,V^{ij,r}_t)_{t\in [0,T]}$
    which satisfies
    \begin{equation*}
		U^{ij,r}_t=0,\quad V^{ij,r}_t=0,\quad\text{for}\quad 0\leq t<r\leq T,
	\end{equation*}
	and is the unique solution in $\mathcal{S}^2(\mathbb{R}^{l})\times\mathcal{H}^2(\mathbb{R}^{l\times d})$ of the BSDE
	\begin{align*}
		U^{j,r}_t&=D^j_r\xi+\int_t^T\partial_{y}g_s(Y_s,Z_s)U^{j,r}_s+\partial_{z}g_s(Y_s,Z_s)V^{j,r}_s+D^j_rg_s(Y_s,Z_s)ds -\int_t^{T}V^{j,r}_sdW_s.
	\end{align*}
	Applying It\^{o}'s formula to $|U^{j,r}_t|^2$ yields
	\begin{align*}
		|U^{j,r}_t|^2&=|D^j_r\xi|^2-\int_t^{T}2U^{j,r}_sV^{j,r}_sdW_s\\
		&\quad +\int_t^T2U^{j,r}_s\partial_{y}g_s(Y_s,Z_s)U^{j,r}_s+2U^{j,r}_s\partial_{z}g_s(Y_s,Z_s)V^{j,r}_s+2U^{j,r}_sD^j_rg_s(Y_s,Z_s)-|V^{j,r}_s|^2ds\\
		&\leq |D^j_r\xi|^2-\int_t^{T}2U^{j,r}_sV^{j,r}_sdW_s+\int_t^T 2B|U^{j,r}_s|^2+2\rho|U^{j,r}_s||V^{j,r}_s|+2\sqrt{\sum_{i=1}^{l}q^2_{ij}(s)}|U^{j,r}_s|-|V^{j,r}_s|^2ds\\
		&\leq |D^j_r\xi|^2-\int_t^{T}2U^{j,r}_sV^{j,r}_sdW_s+\int_t^T \left(2B+\rho^2+1\right)|U^{j,r}_s|^2+\sum_{i=1}^{l}q^2_{ij}(s)ds.
	\end{align*}
	Using condition \ref{b3} and taking conditional expectation in the above inequality yields
	\begin{equation}
	\label{eq:boundZU}
		|U^{j,r}_t|^2\leq E\Big[\sum_{i=1}^{l}A^2_{ij}+\int_t^T\left(2B+\rho^2+1\right)|U^{j,r}_s|^2+\sum_{i=1}^{l}q^2_{ij}(s)ds \,\Big| \,\mathcal{F}_t\Big].
	\end{equation}
	By \citet[Proposition 5.3]{karoui01} the process $Z$ is a version of the trace $(U^{t}_t)_{t\in [0,T]}$ of the Malliavin derivative of $Y$. Hence \eqref{eq:boundZlemma} follows from \eqref{eq:boundZU} by applying Gronwall's inequality.
\end{proof}
\begin{proof}[Theorem \ref{thm:sup}]
	Define the Lipschitz continuous function $\tilde{g}$ by
	\begin{align}
	\label{eq:g-tilde}
		\tilde{g}_t(y,z)=
		\begin{cases}
			g_t(y,z)\quad &\text{if}\quad |z|\leq Q,\\
			g_t(y,Qz/|z|)\quad &\text{if}\quad |z|>Q.
		\end{cases}
	\end{align}
	By \citet[Lemma 2.5]{Che-Nam} and \citet[Theorem 5.1]{karoui01} the BSDE corresponding to $(\tilde{g},\xi)$ has a unique solution $(Y,Z)\in\mathcal{S}^4(\mathbb{R}^{l})\times\mathcal{H}^4(\mathbb{R}^{{l}\times d})$.
For $x=(y,z)\in\mathbb{R}^{l+l\times d}$ let $\beta\in C^{\infty}(\mathbb{R}^{l+l\times d})$ be the mollifier
	\begin{align*}
		\beta(x):=
		\begin{cases}
			\lambda\exp\left(-\frac{1}{1-|x|^2}\right)\quad &\text{if}\quad |x|<1,\\
			0\quad &\text{otherwise},
		\end{cases}
	\end{align*}
	where the constant $\lambda\in\mathbb{R}_{+}$ is chosen such that $\int_{\mathbb{R}^{l+l\times d}}\beta(x)dx=1$.
	Set $\beta^{n}(x):=n^{l+l\times d}\beta(nx)$, $n\in\mathbb{N}\setminus\{0\}$, and define
	\begin{align*}
		g^n_t(\omega,x):=\int_{\mathbb{R}^{l+l\times d}}\tilde{g}_t(\omega,x')\beta^n(x-x')dx'
	\end{align*}
	so that for each $n>0$ the function $g^n$ satisfies \ref{b1prime} and \ref{b3prime} with the constant $\rho$ replaced by $\rho(Q)$.
	By Lemma \ref{le:sup} the BSDE corresponding to $(g^n,\xi)$
    has a unique solution $(Y^n,Z^n)$ in $\mathcal{S}^4(\mathbb{R}^{l})\times\mathcal{H}^4(\mathbb{R}^{l\times d})$ which satisfies
	\begin{align*}
		|Z^{n,j}_t|^2&\leq \left(\sum_{i=1}^{l}A^2_{ij}+\sum_{i=1}^{l}\int_t^{T}q^2_{ij}(s)e^{-\left(2B+\rho^2(Q)+1\right)(T-s)}ds\right)e^{\left(2B+\rho^2(Q)+1\right)(T-t)}\\
		&\leq \left(\sum_{i=1}^{l}A^2_{ij}+\sum_{i=1}^{l}\int_0^{T}q^2_{ij}(s)ds\right)e^{\left(2B+\rho^2(Q)+1\right)T}.
	\end{align*}
	Since $T\leq\frac{\log (2)}{2B+\rho^2(Q)+1}$ we obtain
	\begin{equation*}
		|Z^{n,j}_t|^2\leq 2\left(\sum_{i=1}^{l}A^2_{ij}+\sum_{i=1}^{l}\int_0^{T}q^2_{ij}(s)ds\right)\quad \mbox{for all }j = 1, \dots, d.
	\end{equation*}
	This shows $|Z^n_t| \le Q$. Since $g^n$ converges uniformly in $(t,\omega, y,z)$ to $\tilde{g}$, using the procedure of the proof of \citet[Theorem 2.2]{Che-Nam}, it follows that $(Y^n, Z^n)$ converges to $(Y,Z)$ in ${\cal S}^2(\mathbb{R}^l)\times {\cal H}^2(\mathbb{R}^{l\times d})$, so that $|Z_t|\le Q$.
	Since $\tilde{g}(y,z)=g(y,z)$ for all $(y,z)\in \mathbb{R}^l\times\mathbb{R}^{l\times d}$ with $|z|\le 	Q$,
	it follows that $(Y,Z)$ is the unique solution of the BSDE corresponding to $(\xi, g)$ in $\mathcal{S}^{4}(\mathbb{R}^{l})\times\mathcal{S}^{\infty}(\mathbb{R}^{{l}\times d})$.
\end{proof}
\begin{corollary}
\label{thm:corol-Malliavin}
	Suppose \ref{b1}-\ref{b3} hold, $T\leq \frac{\log (2)}{2B+\rho^2(Q)+1}$ and $(Y,Z) \in \mathcal{S}^{4}(\mathbb{R}^{l})\times\mathcal{S}^{\infty}(\mathbb{R}^{{l}\times d})$ is the solution of the BSDE \eqref{sup}. Then $Y_t\in {\cal D}^{1,2}(\mathbb{R}^l)$ for all $t\in[0,T]$ and for every $j = 1, \dots, d$, one has
	\begin{equation}
	\label{eq:bound_corol}
		|D_r^j Y_t|^2\leq 2\Bigg(\sum_{i=1}^{l}A^2_{ij}+\sum_{i=1}^{l}\int_0^{T}q^2_{ij}(s)ds\Bigg)\quad\mbox{for all }r\in[0,t].
\end{equation}
\end{corollary}
\begin{proof}
	Since $|Z|\leq Q$ is bounded, $(Y,Z)$ solves the BSDE with terminal condition $\xi$ and generator $\tilde{g}$ defined by \eqref{eq:g-tilde}.
	If $g$ satisfies \ref{b1prime} and \ref{b3prime}, then the result follows from Lemma \ref{le:sup}.
	Otherwise consider the sequence of smooth functions $g^n$ converging to $g$ as defined in the proof of Theorem \ref{thm:sup}.
	Let  $(Y^n,Z^n)\in\mathcal{S}^4(\mathbb{R}^{l})\times\mathcal{H}^4(\mathbb{R}^{l\times d})$ be the solutions to the BSDEs corresponding to $(g^n,\xi)$, which converge to $(Y,Z)$ in $\mathcal{S}^2(\mathbb{R}^{l})\times\mathcal{H}^2(\mathbb{R}^{l\times d})$.
	By Lemma \ref{le:sup} $(Y^n_t,Z^n_t) \in {\cal D}^{1,2}(\mathbb{R}^l)\times {\cal D}^{1,2}(\mathbb{R}^{l\times d})$ for each $t \in [0,T]$
	and the arguments in the proof of Theorem \ref{thm:sup} imply
	\begin{equation*}
		|D_r^j Y^{n}_t|^2\leq 2\left(\sum_{i=1}^{l}A^2_{ij}+\sum_{i=1}^{l}\int_0^{T}q^2_{ij}(s)ds\right)\quad j = 1, \dots, d, \,\, r,t \in [0,T].
	\end{equation*}
	Hence, $\sup_{n \in \mathbb{N}}E[\int_0^T|D_r^j Y^n_t|^2\,dr]< \infty$ for each $t \in [0,T]$.
	Since $(Y^n_t)$ converges to $Y_t$ in $L^2$, it follows from \citet[Lemma  1.2.3]{Nualart2006} that $Y_t \in {\cal D}^{1,2}(\mathbb{R}^l)$ and $(DY^n_t)$ converges to  $DY_t$ in the weak topology of ${\cal H}^2(\mathbb{R}^{l\times d})$.
	Thus, $D_rY_t$ satisfies \eqref{eq:bound_corol}.
\end{proof}
As a concequence to Theorem \ref{thm:sup}, we give a condition for global solvability of fully coupled systems of BSDEs.
For the remainder of this section we put
\begin{equation*}
	\Delta_n := \frac{\log (2)}{2B+\rho^2(2^nQ)+1},\quad n \in \mathbb{N}.
\end{equation*}

\begin{proposition}
\label{thm:prob-global}
	Assume that  \ref{b1}-\ref{b2} hold, that there exists $N\in\mathbb{N}$ such that $\sum_{n=0}^{N}\Delta_n\geq T$, and
	\ref{b3} holds 
	with $Q$ replaced by $2^NQ$.
	Then the BSDE \eqref{sup} has a unique solution in $\mathcal{S}^{4}(\mathbb{R}^{l})\times\mathcal{S}^{\infty}(\mathbb{R}^{{l}\times d})$ and $|Z_t|\leq 2^{N}Q$.
\end{proposition}
\begin{proof}
	If $T\le \Delta_0$ then the result follows from Theorem \ref{thm:sup}.
	Otherwise, if $T > \Delta_0$ it follows by the same arguments as in the proof of
	Theorem \ref{thm:sup} that the BSDE \eqref{sup} has a unique solution $(Y^0,Z^0)$ in $\mathcal{S}^{4}(\mathbb{R}^{l})\times\mathcal{S}^{\infty}(\mathbb{R}^{{l}\times d})$ on the interval $[T-\Delta_0,T]$.
	Moreover, $Z^0$ satisfies $|Z_t^0|\leq Q$ and by Corollary \ref{thm:corol-Malliavin} one has
    $Y^0_{T-\Delta_0} \in {\cal D}^{1,2}(\mathbb{R}^l)$ and for every $r \le T-\Delta_0$,
	\begin{equation*}
		|D^j_rY^0_{T-\Delta_0 }|^2\leq \sum_{i=1}^{l}2|A_{ij}|^2+\sum_{i=1}^{l}\int_0^{T}2|q_{ij}(t)|^2dt\quad \mbox{for all }j = 1, \dots, d.
	\end{equation*}
	Since $g$ satisfies \ref{b3} for all $(y,z)\in \mathbb{R}^l\times \mathbb{R}^{l\times d}$ such that $|z|\le c Q$,
	again by Theorem \ref{thm:sup} the BSDE \eqref{sup} with terminal condition $Y^0_{T-\Delta_0}$ has a unique solution $(Y^1, Z^1)$ in $\mathcal{S}^{4}(\mathbb{R}^{l})\times\mathcal{S}^{\infty}(\mathbb{R}^{{l}\times d})$ on $[(T-\Delta_0 - \Delta_1)\vee 0,T-\Delta_0]$ , and
	\begin{align*}
		&|D^j_rY^1_{(T-\Delta_0 - \Delta_1)\vee 0 }|^2\leq \sum_{i=1}^{l}2^2|A_{ij}|^2+\sum_{i=1}^{l}\int_0^{T}(2^2+2)|q_{ij}(t)|^2dt, \quad \mbox{for all }j = 1, \dots, d\\
		&|Z_t^1|\leq 2 Q, \quad t\in \left[(T-\Delta_0 - \Delta_1)\vee 0,T-\Delta_0 \right].
	\end{align*}
	Repeating the previous arguments, for $N\ge 2$ the BSDE \eqref{sup} has a unique solution $(Y^N, Z^N)$ in $\mathcal{S}^{4}(\mathbb{R}^{l})\times\mathcal{S}^{\infty}(\mathbb{R}^{{l}\times d})$ on $[(T-\sum_{n=0}^{N}\Delta_n)\vee 0 ,(T-\sum_{n=0}^{N-1}\Delta_n )\vee 0]$ with terminal condition $Y_{(T-\sum_{n=0}^{N-1}\Delta_n)\vee 0}$. Moreover,
	\begin{align*}
	&|D^j_rY^N_{(T-\sum_{n=0}^{N}\Delta_n)\vee 0}|^2\leq \sum_{i=1}^{l}2^{N}|A_{ij}|^2+\sum_{i=1}^{l}\int_0^{T}(\sum_{k=1}^{N}2^{k})|q_{ij}(t)|^2dt  \quad \mbox{for all }j = 1, \dots, d\\
	&|Z^N_t|\leq 2^{N}Q, \quad t\in \Big[(T-\sum_{n=0}^{N}\Delta_n)\vee 0,(T-\sum_{n=0}^{N-1}\Delta_n)\vee 0 \Big].
	\end{align*}
	Hence, the pair $(Y, Z)$ given by
	\begin{align*}
		Y &:= Y^01_{\left[T-\Delta_1, T\right]} + \sum_{n=1}^NY^n1_{\left[\left(T-\sum_{i=0}^n\Delta_i \right)\vee 0, (T -\sum_{i=0}^{n-1}\Delta_i)\vee 0 \right]}\\
		Z &:= Z^01_{\left[T-\Delta_1, T\right]} + \sum_{n=1}^NZ^n1_{\left[\left(T-\sum_{i=0}^n\Delta_i \right)\vee 0, (T -\sum_{i=0}^{n-1}\Delta_i)\vee 0 \right]}
	\end{align*}
	solves \eqref{sup} and its uniqueness follows from Theorem \ref{thm:sup}.
\end{proof}
\begin{remark}
	The condition $\sum_{n= 0}^N\Delta_n\ge T$ for some $N \in \mathbb{N}$ does not guarantee global solvability of multidimentional BSDEs with superquadratic growth.
	In fact, if $\rho(x) \ge C(1+\sqrt{x})$ for all $x\ge0 $, then $\sum_{n\ge 0}\Delta_n < \infty$.
	However, it does guarantee global solvability for BSDEs whose generator grows slightly faster than the linear function.
	For instance, if $\rho(x) \le C(1 + \sqrt{\log(1+x)})$ one has
\begin{align*}
\sum_{n=0}^{\infty}\frac{\log(2)}{2B+\rho^2(2^{N}Q)+1}&\geq\sum_{n=0}^{\infty}\frac{\log(2)}{2B+2C^2(1+\log(2^{N}(1+Q)))+1}\\
&=\sum_{n=0}^{\infty}\frac{\log(2)}{2B+2C^2(1+\log(1+Q)+n\log(2))+1}=\infty.
\end{align*}
\end{remark}

\section{Coupled FBSDE with superquadratic growth}\label{sec:FBSDE}
\subsection{Proof of Theorem \ref{thm:markov1}}


\emph{Step 1}: We first assume that $h,b$ and $g$ are continuously differentiable in all variables.
	Let us define 	
\begin{align*}
	C^1_{k,\lambda,l,d} :=\frac{k^2_5}{k^2_3}\wedge\frac{\log 2}{k_1}\wedge\frac{\lambda_2}{k_2M}\wedge\frac{\log 2}{2k_4+\rho^2(M)+1}
	\end{align*}
	with $M:=4k_5\lambda_2\sqrt{dl}$.		
	We will show that for $T\le C^1_{k,\lambda,l,d}$, the sequence $(X^n, Y^n, Z^n)$ given by $X^0 = 0$, $Y^0 = 0$, $Z^0 =0$ and
	\begin{equation*}
		\begin{cases}
			X^{n+1}_t &= x + \int_0^tb(X^{n+1}_u, Y^n_u)\,du + \int_0^t\sigma_u\,dW_u\\
			Y^{n+1}_t &= h(X^{n+1}_T) + \int_t^Tg_u(X^{n+1}_u, Y^{n+1}_u, Z^{n+1}_u)\,du - \int_t^TZ^{n+1}_u\,dW_u, \quad n\ge 1
		\end{cases}
	\end{equation*}
	is well defined and that $\abs{Z^{n}_t} \le M$ for all $n\in \mathbb{N}$ and $t \in [0,T]$.
	The process $X^1$ is well defined, $X^1_t$ belongs to ${\cal D}^{1,2}(\mathbb{R}^m)$ for every $t$ and the process $(D_rX_t)_{t\in [0,T]}$ satisfies the linear equation
	\begin{align*}
        &D_rX_t^{1}=0,~0\le t<r\le T,\\
		&D_rX_t^{1} = \int_r^t(\partial_x bD_rX^{1}_u + \partial_ybD_rY_u^{0})\,du + D_r\left( \int_r^t\sigma_u\,dW_u \right),~0\le r\le t\le T,
	\end{align*}
	 with $D_r( \int_r^t\sigma_u\,dW_u ) =\sigma1_{[r,t]} $, see \citet[Lemma 2.2.1 and Theorem 2.2.1]{Nualart2006}.
	 Hence, since $b$ is Lipschitz continuous, we have
	 \begin{align*}
		\abs{D_rX^{1}_t} &\le \int_r^tk_1D_rX^{1}_u \,du + \sigma_r \quad \text{and} \quad
		  \abs{D_rX^{1}_t} \le\lambda_2e^{Tk_1},
	\end{align*}
	where the second estimate comes from Gronwall's inequality.
	We will now show that since $T\le C^1_{k,\lambda,l,d}$, $h(X^1_T)$ and $g(X^1,\cdot,\cdot)$ satisfy \ref{b1}-\ref{b3}.
	In fact, since $h$ is continuously differentiable and $X^1_T \in {\cal D}^{1,2}(\mathbb{R}^m)$, it follows from the chain rule, see for instance \citet[Proposition 1.2.4]{Nualart2006}, that $h(X^1_T) \in {\cal D}^{1,2}(\mathbb{R}^{l})$ and $|D^j_r(h(X_T^1))| = |\partial_xh(X^1_T)D^j_rX^1_T| \le \lambda_2k_5e^{Tk_1}$ for all $r \in [0,T]$, $j= 1, \dots, d$.
	Using $T\le \frac{\log 2}{k_1}$, we deduce that $h(X_T^1)$ satisfies \ref{b2} with $A_{ij}:= 2\lambda_2k_5$.
	Similarly, by \ref{a4} and using that the function $x\mapsto g(x,y,z)$ is continuously differentiable, it follows that $g_.(X^1_.,y,z) \in {\cal L}^{1,2}_a(\mathbb{R}^l)$ and $\abs{D^j(g^i_.(X^1,y,z))} \le \lambda_2k_3e^{Tk_1}$, $j= 1,\dots,d$ for all $(y,z)\in \mathbb{R}^l\times \mathbb{R}^{l\times d}$ such that $|z| \le M$ and, due to \ref{a5}, applying the same argument to $\hat{g}_t(x,y,y',z,z'):= g_t(x,y,z)-g_t(x,y',z')$ yields
	\begin{equation*}
	|D_r^jg_t(X^1_t,y,z) - D^j_rg_t(X^1_t,y',z')| \le K\lambda_2e^{Tk_1}.
	\end{equation*}
	Using $T\le \frac{k^2_5}{k^2_3}\wedge\frac{\log 2}{k_1}$, we deduce that $g_.(X^1_.,y,z)$ satisfies \ref{b3} with $q_{ij}=2\lambda_2k_3$ and $K_u(t):= 2K\lambda_2$.
	Moreover due to \ref{a4}, the function $(t,y,z)\mapsto g_t(X^1_t,y,z)$ satisfies \ref{b1}.

	Therefore, by $T\le \frac{\log 2}{2k_4+\rho^2(M)+1}$, Theorem \ref{thm:sup} ensures that $(Y^1, Z^1)$ exists.
	Consider the function $\tilde{g}$ defined by
	\begin{equation*}
		\tilde{g}_t(x,y,z) = \begin{cases}
			g_t(x,y,z) & \text{ if } |z|\le M\\
			g_t(x,y,zM/|z| ) & \text{ if } |z| >M.
		\end{cases}
	\end{equation*}	
	Since $(Y^1, Z^1)$ also solves the BSDE with terminal condition $h(X^1_T)$ and a Lipschitz generator $\tilde{g}(X^1, \cdot,\cdot)$, it follows from Lemma \ref{le:sup} and its proof that  $(Y^1_t, Z^1_t)\in{\cal D}^{1,2}(\mathbb{R}^{l})\times{\cal D}^{1,2}(\mathbb{R}^{l\times d})$ for all $t \in [0,T]$ and  $D_tY^1$ is bounded and it holds $Z^1_t = D_tY^1_t$.
	In addition, we have $\abs{D_rX^1_t}\le 4\lambda_2$ and $\abs{D_rY^1_t}\le M$.

	Now let $n \in \mathbb{N}$, assume that $(X^n_t, Y^n_t, Z^n_t)\in{\cal D}^{1,2}(\mathbb{R}^{m})\times{\cal D}^{1,2}(\mathbb{R}^{l})\times{\cal D}^{1,2}(\mathbb{R}^{l\times d})$, $Z^n_t = D_tY^n_t$ and $\abs{D_rX^n_t}\le 4\lambda_2$, $\abs{D_rY_t^n} \le M$ for all $r,t \in [0,T]$.
	The process $X^{n+1}$ is well defined, for each $t$; $X^{n+1}_t$ belongs to ${\cal D}^{1,2}(\mathbb{R}^m)$ and it holds
	\begin{align*}
        &D_rX_t^{n+1}=0,~0\le t<r\le T,\\
		&D_rX_t^{n+1} = \sigma_r+\int_r^t(\partial_x bD_rX^{n+1}_u + \partial_ybD_rY_u^{n})\,du,~0\le r\le t\le T.
	\end{align*}
	 Since $\partial_xb$, $\partial_yb$ and $\sigma$ are bounded by $k_1$, $k_2$ and $\lambda_2$ respectively, it follows from Gronwall's inequality that
	\begin{equation*}
		\abs{D_rX^{n+1}_t} \le e^{Tk_1}\left( \lambda_2 + k_2\int_0^T\abs{D_rY^n_u}\,du \right).
	\end{equation*}
	Hence,
	\begin{equation}
	\label{eq:bound_DX_Markov}
		\abs{D_rX^{n+1}_t} \le e^{Tk_1}\left( \lambda_2 + k_2TM \right) < \infty
	\end{equation}
	so that since $T \le \frac{\lambda_2}{k_2M}$, we have $\abs{D_rX^{n+1}_t} \le 4\lambda_2$.
	As above, $h(X^{n+1}_T)$ and $g_.(X^{n+1}, y,z)$ are Malliavin differentiable and satisfy \ref{b1}-\ref{b3} with $A_{ij}:= 2\lambda_2k_5$, $q_{ij}=2\lambda_2k_3$ and $K_u(t):= 2K\lambda_2$.
	It then follows again from Theorem \ref{thm:sup} that $(Y^{n+1}, Z^{n+1})$ exists and $|Z^{n+1}|\le M$ is bounded.
	Since $(Y^{n+1}, Z^{n+1})$ also solves the BSDE with terminal condition $h(X^{n+1}_T)$ and a Lipschitz generator $\tilde{g}(X^{n+1}, \cdot,\cdot)$, Lemma \ref{le:sup} and its proof guarantee that  $(Y^{n+1}_t, Z^{n+1}_t)\in{\cal D}^{1,2}(\mathbb{R}^{l})\times{\cal D}^{1,2}(\mathbb{R}^{l\times d})$ for all $t \in [0,T]$ and  $D_tY^{n+1}$ is bounded and it holds $Z^{n+1}_t = D_tY^1_t$, with $\abs{D_rY^{n+1}_t}\le M$.
	
	\emph{Step 2}: Now we show that there is a positive constant $2^2_{k,\lambda,l,d}$ such that if $T\le 2^2_{k,\lambda,l,d}$, then $(X^n,Y^n,Z^n)$ is a Cauchy sequence in $\mathcal{S}^2(\mathbb{R}^m)\times\mathcal{S}^2(\mathbb{R}^{l})\times\mathcal{H}^2(\mathbb{R}^{l\times d})$.
	Using \ref{a1} we can estimate the norm of the difference $X^{n+1}_t - X^n_t$ as
	\begin{align*}
		&|X^{n+1}_t-X^{n}_t|^2\leq 2\left(\int_0^tk_1|X^{n+1}_s-X^{n}_s|ds\right)^2+2\left(\int_0^tk_2|Y^{n}_s-Y^{n-1}_s|ds\right)^2.
	\end{align*}
	Thus
	\begin{align*}
		&\sup_{0\leq t\leq T}|X^{n+1}_t-X^{n}_t|^2\leq 2\left(\int_0^Tk_1|X^{n+1}_s-X^{n}_s|ds\right)^2+2\left(\int_0^Tk_2|Y^{n}_s-Y^{n-1}_s|ds\right)^2.
	\end{align*}
	Taking expectation on both sides and using Cauchy-Schwarz' inequality, we have
	\begin{align*}
		&E\left[\sup_{0\leq t\leq T}|X^{n+1}_t-X^{n}_t|^2\right]\leq 2Tk^2_1E\left[\int_0^T|X^{n+1}_s-X^{n}_s|^2ds\right]+2Tk^2_2E\left[\int_0^T|Y^{n}_s-Y^{n-1}_s|^2ds\right]\\
		&\leq2T^2k^2_1E\left[\sup_{0\leq t\leq T}|X^{n+1}_t-X^{n}_t|^2\right]+2T^2k^2_2E\left[\sup_{0\leq t\leq T}|Y^{n}_t-Y^{n-1}_t|^2\right].
	\end{align*}
	Choosing $T$ to be small enough so that $2T^2k^2_1\leq\frac{1}{2}$, it follows
	\begin{equation}
	\label{eq:estimate_X}
		E\left[\sup_{0\leq t\leq T}|X^{n+1}_t-X^{n}_t|^2\right] \leq4T^2k^2_2E\left[\sup_{0\leq t\leq T}|Y^{n}_t-Y^{n-1}_t|^2\right].
	\end{equation}
	On the other hand, applying It\^{o}'s formula to $e^{\beta t}|Y^{n+1}_t-Y^{n}_t|^2$, $\beta \ge 0$, we have
	\begin{align*}
		e^{\beta t}|Y^{n+1}_t-Y^{n}_t|^2&=e^{\beta T}|h(X^{n+1}_T)-h(X^n_T)|^2-2\int_t^Te^{\beta s}(Y^{n+1}_s-Y^{n}_s)(Z^{n+1}_s-Z^{n}_s)dW_s\\
		&\quad -\int_t^Te^{\beta s}(Z^{n+1}_s-Z^{n}_s)^2ds-\int_t^T\beta e^{\beta s}(Y^{n+1}_s-Y^{n}_s)^2ds\\
		&\quad+2\int_t^Te^{\beta s}(Y^{n+1}_s-Y^{n}_s)\left[g_s(X^{n+1}_s,Y^{n+1}_s,Z^{n+1}_s)-g_s(X^{n}_s,Y^{n}_s,Z^{n}_s)\right]ds.
	\end{align*}
	Hence, due to the condition \ref{a3} and the boundedness of $(Z^n)$, it holds
	\begin{align*}
		&e^{\beta t}|Y^{n+1}_t-Y^{n}_t|^2+\int_t^Te^{\beta s}(Z^{n+1}_s-Z^{n}_s)^2ds\\
		&\leq e^{\beta T}\abs{h(X^{n+1}_T)-h(X^n_T)}^2-2\int_t^Te^{\beta s}(Y^{n+1}_s-Y^{n}_s)(Z^{n+1}_s-Z^{n}_s)dW_s\\
		&\quad -\int_t^T\beta e^{\beta s}(Y^{n+1}_s-Y^{n}_s)^2ds+2\int_t^Te^{\beta s}\rho(M)\abs{Y^{n+1}_s-Y^{n}_s}\abs{Z^{n+1}_s-Z^{n}_s}ds\\
		&\quad 
		+2\int_t^Te^{\beta s}k_{7}\abs{Y^{n+1}_s-Y^{n}_s}\abs{X^{n+1}_s-X^{n}_s}ds+2\int_t^Te^{\beta s}k_4\abs{Y^{n+1}_s-Y^{n}_s}^2ds.
	\end{align*}
	With some positive constants $\alpha_1$, $\alpha_2$, it follows from \ref{a3} and Young's inequality that
	\begin{align}
	\nonumber &e^{\beta t}|Y^{n+1}_t-Y^{n}_t|^2+\int_t^Te^{\beta s}(Z^{n+1}_s-Z^{n}_s)^2ds \leq e^{\beta T}k^2_5|X^{n+1}_T-X^n_T|^2\\
	\nonumber &\quad-2\int_t^Te^{\beta s}(Y^{n+1}_s-Y^{n}_s)(Z^{n+1}_s-Z^{n}_s)dW_s +\alpha_2\int_t^Te^{\beta s}|X^{n+1}_s-X^{n}_s|^2ds\\
	\label{eq:firstestimate_YZ}&\quad+\left(\frac{(\rho(M))^2}{\alpha_1}+\frac{k^2_{3}}{\alpha_2}+2k_4-\beta\right)\int_t^T e^{\beta s}(Y^{n+1}_s-Y^{n}_s)^2ds+\alpha_1\int_t^Te^{\beta s}|Z^{n+1}_s-Z^{n}_s|^2ds.
	\end{align}
	Letting $\beta=\frac{(\rho(M))^2}{\alpha_1}+\frac{k^2_{7}}{\alpha_2}+2k_8$ and taking expectation on both sides above, we have
	\begin{multline*}
	E\left[e^{\beta t}|Y^{n+1}_t-Y^{n}_t|^2\right]+E\left[\int_t^Te^{\beta s}(Z^{n+1}_s-Z^{n}_s)^2ds\right]\leq e^{\beta T}k^2_5E\left[|X^{n+1}_T-X^n_T|^2\right]\\
	+\alpha_1E\left[\int_t^Te^{\beta s}|Z^{n+1}_s-Z^{n}_s|^2ds\right] +\alpha_2E\left[\int_t^Te^{\beta s}|X^{n+1}_s-X^{n}_s|^2ds\right].
	\end{multline*}
	Putting $\alpha_1=\frac{1}{2}$ and $\alpha_2=1$, the previous estimate yields
	\begin{equation*}
	E\left[\int_0^Te^{\beta s}(Z^{n+1}_s-Z^{n}_s)^2ds\right]\leq 2e^{\beta T}k^2_5E\left[|X^{n+1}_T-X^n_T|^2\right]  +2E\left[\int_0^Te^{\beta s}|X^{n+1}_s-X^{n}_s|^2ds\right].
	\end{equation*}
	Next, taking conditional expectation with respect to $\mathcal{F}_t$ in \eqref{eq:firstestimate_YZ} gives
	\begin{multline*}
		e^{\beta t}|Y^{n+1}_t-Y^{n}_t|^2+E\left[\int_t^Te^{\beta s}(Z^{n+1}_s-Z^{n}_s)^2ds\bigg|\mathcal{F}_t\right]\leq e^{\beta T}k^2_5E\left[|X^{n+1}_T-X^n_T|^2|\mathcal{F}_t\right]\\
		+\alpha_1E\left[\int_t^Te^{\beta s}|Z^{n+1}_s-Z^{n}_s|^2ds\bigg|\mathcal{F}_t\right] +\alpha_2E\left[\int_t^Te^{\beta s}|X^{n+1}_s-X^{n}_s|^2ds\bigg|\mathcal{F}_t\right].
	\end{multline*}
	Thus, by Burkholder-Davis-Gundy's inequality, with a positive constant $c_1$ and $\alpha_1 = \frac{1}{2}$, $\alpha_2 = 1$, we have
	\begin{align*}
		&E\left[\sup_{0\leq t\leq T}e^{\beta t}|Y^{n+1}_t-Y^{n}_t|^2\right]
		\leq c_1e^{\beta T}k^2_5E\left[|X^{n+1}_T-X^n_T|^2\right]\\
		&\qquad \qquad\qquad+c_1\frac{1}{2}E\left[\int_0^Te^{\beta s}|Z^{n+1}_s-Z^{n}_s|^2ds\right] +c_1E\left[\int_0^Te^{\beta s}|X^{n+1}_s-X^{n}_s|^2ds\right]\\
		&\qquad\qquad \quad\leq 2c_1e^{\beta T}k^2_5E\left[|X^{n+1}_T-X^n_T|^2\right] +2c_1E\left[\int_0^Te^{\beta s}|X^{n+1}_s-X^{n}_s|^2ds\right].
	\end{align*}
	It now follows from \eqref{eq:estimate_X} that
	\begin{align*}
	&E\left[\sup_{0\leq t\leq T}|Y^{n+1}_t-Y^{n}_t|^2\right]+E\left[\int_0^T(Z^{n+1}_s-Z^{n}_s)^2ds\right]\\
	&\qquad \leq 8(c_1+1)e^{\beta T}(k^2_5+T)T^2k^2_2E\left[\sup_{0\leq t\leq T}|Y^{n}_t-Y^{n-1}_t|^2\right].
	\end{align*}
	Taking $T$ small enough so that
	\begin{equation*}
	8(c_1+1)e^{\beta T}(k^2_5+T)T^2k^2_2\leq\frac{1}{2},
	\end{equation*}
	we obtain that $(X^n,Y^n,Z^n)$ is a Cauchy sequence in $\mathcal{S}^2(\mathbb{R}^m)\times\mathcal{S}^2(\mathbb{R}^{l})\times\mathcal{H}^2(\mathbb{R}^{l\times d})$.
	Thus, it suffices to define $2^2_{k,\lambda,d,l}$ by the conditions
	\begin{equation*}
		\begin{cases}
			2T^2k_1^2\le \frac{1}{2}\\
			8(c_1+1)e^{\beta T}(k^2_5+T)T^2k^2_2\leq\frac{1}{2}.
		\end{cases}
	\end{equation*}
	By continuity of $b,g$ and $h$ we have the existence of a solution $(X,Y,Z)$ in $\mathcal{S}^2(\mathbb{R}^m)\times\mathcal{S}^2(\mathbb{R}^{l})\times\mathcal{H}^2(\mathbb{R}^{l\times d})$ of FBSDE \eqref{eq:fbsde1} and it follows from the boundedness of $(Z^n)$ that $\abs{Z_t}\le M$.
	The uniqueness in $\mathcal{S}^2(\mathbb{R}^m)\times\mathcal{S}^2(\mathbb{R}^{l})\times\mathcal{S}^\infty(\mathbb{R}^{l\times d})$ follows from the boundedness of $Z$ and by repeating the above arguments on the difference of two solutions.

	\emph{Step 3}: If one of the functions $b$, $g$ or $h$ is not differentiable, we apply the technique of the proof of Theorem \ref{thm:sup}.
	Namely, we use an approximation by the smooth functions defined as follows:
	For $n\in\mathbb{N}$, let $\beta^1_n, \beta^2_n$ and $\beta^3_n$ be nonnegative $C^\infty$ functions with support on $\{x\in\mathbb{R}^m:|x|\leq\frac{1}{n}\}$, $\{x\in\mathbb{R}^{m+l}:|x|\leq\frac{1}{n}\}$ and $\{x\in\mathbb{R}^{m+l+l\times d}:|x|\leq\frac{1}{n}\}$ respectively, and satisfying $\int_{\mathbb{R}^m}\beta^1_n(r)dr=1$, $\int_{\mathbb{R}^{m+l}}\beta^2_n(r)dr=1$ and $\int_{\mathbb{R}^{m+l+l\times d}}\beta^3_n(r)dr=1$.
	We define the convolutions
	\begin{align*}
		b^n_t(x,y)&:=\int_{\mathbb{R}^{m+l}}b_t(x',y')\beta^2_n(x'-x,y'-y)dx'dy',\quad
		h^n(x):=\int_{\mathbb{R}^{m}}h(x')\beta^1_n(x'-x)dx',\\
		g^n(u,x,y,z)&:=\int_{\mathbb{R}^{m+l+l\times d}}g(u,x',y',z')\beta^3_n(x'-x,y'-y,z'-z)dx'dy'dz'.
	\end{align*}
	It is easy to check that $b^n$ satisfies \ref{a1} with the constants $k_1,k_2$ and $2\lambda_1$ and that $g^n$ and $h^n$ satisfy \ref{a4} - \ref{a5} and \ref{a3}, respectively, with the same constants.
	From \emph{Steps 1 and 2}, there exists a positive constant $\bar{C}_{k, \lambda, l,d}$ independent of $n$ such that if $T \le \bar{C}_{k,\lambda,l,d}$, FBSDE \eqref{eq:fbsde1} with parameters $(b^n,h^n,g^n)$ admits a unique solution $(X^n, Y^n, Z^n) \in {\cal S}^2(\mathbb{R}^m)\times {\cal S}^2(\mathbb{R}^{l})\times {\cal S}^\infty(\mathbb{R}^{l\times d})$ and
	\begin{equation*}
		|Z^{n}_t|\leq M.
	\end{equation*}
	By the Lipschitz continuity conditions on $b$ and $h$ and the locally Lipschitz condition of $g$, the sequences $(b^n)$ and $(h^n)$ converge uniformly to $b$ and $h$ on $\mathbb{R}^{m+l}$ and $\mathbb{R}^m$, respectively, and $(g^n)$ converges to $g$ uniformly on $\mathbb{R}^{m+l}\times \Lambda$ for any compact subset $\Lambda$ of $\mathbb{R}^{l\times d}$.
	Combining these uniform convergences with the boundedness of $Z^n$, similar to above, we can show that there exists a constant $\tilde{C}_{k,\lambda,l,d}$ depending only on $k_1,k_2, k_3,k_4,k_5,\lambda_2, l,d$ such that if $T \le \tilde{C}_{k,\lambda,l,d}$, $(X^n, Y^n, Z^n)$ is a Cauchy sequence in the Banach space ${\cal S}^2(\mathbb{R}^m)\times {\cal S}^2(\mathbb{R}^{l})\times {\cal H}^2(\mathbb{R}^{l\times d})$.

	In fact, for any $m,n\in \mathbb{N}$, using Cauchy-Schwarz' inequality we have
	\begin{align*}
		\abs{X^n_t - X^m_t}^2 &\le T\int_0^T\abs{b_u^n(X_u^n,Y^n_u)-b_u^m(X^m_u,Y^m_u)}^2\,du.
	\end{align*}
	Thus, taking the supremum with respect to $t$ and then expectation on both sides give
	\begin{align}
	\nonumber	&\norm{X^n - X^m}^2_{{\cal S}^2(\mathbb{R}^m)} \\
	\nonumber	&\quad \le 3T\int_0^T(\abs{b_u^n(X_u^n,Y_u^n)-b_u(X^n_u,Y^n_u)}^2 + \abs{b_u^m(X_u^m,Y^m_u)-b_u(X^m_u,Y_u^m)}^2\\
	\nonumber	&\quad \quad + \abs{b_u(X_u^n,Y^n_u)-b_u(X^m_u,Y^m_u)}^2)du\\
	\nonumber	&\quad \le 3T\int_0^T(\abs{b_u^n(X_u^n,Y^n_u)-b_u(X^n_u,Y^n_u)}^2 + \abs{b_u^m(X_u^m,Y^m_u)-b_u(X^m_u,Y^m_u)}^2)\,du\\
	\label{eq:estimX}	&\quad \quad + 3k_1^2T^2\norm{X^n-X^m}^2_{{\cal S}^2(\mathbb{R}^m)} + 3k^2_2T^2\norm{Y^n-Y^m}^2_{{\cal S}^2(\mathbb{R}^l)}
	\end{align}
	where the second inequality follows from \ref{a1}.
	On the other hand, applying It\^o's formula as in \emph{Step 2}, one has
	\begin{align}
	\nonumber		&\abs{Y^m_t - Y^n_t}^2+\int_t^T\abs{Z^n_u - Z^m_u}^2\,du\\
	\nonumber	&\quad \le  \abs{h^n(X^n_T) - h(X^n_T)}^2 +  \abs{h^m(X^m_T) - h(X^m_T)}^2 + k^2_5\abs{X^n_T - X^m_T}^2 \\
	\nonumber	&\quad \quad -2\int_t^T(Y^{n+1}_s-Y^{n}_s)(Z^{n+1}_s-Z^{n}_s)dW_s \\
	\nonumber	&\quad \quad + \int_t^T\abs{Y^n_u - Y^m_u}(\abs{g^n_u(X^n_u,Y^n_u,Z^n_u)-g_u(X^n_u,Y^n_u,Z^n_u)} \\
	\nonumber	&\qquad \qquad + \abs{g^m_u(X^m_u,Y^m_u,Z^m_u)-g_u(X^m_u,Y^m_u,Z^m_u)} + k_3\abs{X^n_u-X^m_u} + k_4\abs{Y^n_u-Y^m_u}\\
	\label{eq:estimCauchy}	&\qquad \qquad + \rho(M)\abs{Z^n_u-Z^m_u}).
	\end{align}
	Taking expectation, due to Young's inequality we have
	\begin{align*}
	 &\norm{Z^{n}-Z^m}^2_{{\cal H}^2(\mathbb{R}^{l\times d})}\\
	  &\quad \le  E[\abs{h^n(X^n_T) - h(X^n_T)}^2] +  E[\abs{h^m(X^m_T) - h(X^m_T)}^2] + (k^2_5+ \frac{1}{2}Tk^2_3)\norm{X^n - X^m}_{{\cal S}^2(\mathbb{R}^m)}^2 \\
	\nonumber	&\quad\quad + \frac{1}{2}T\norm{Y^n-Y^m}^2_{{\cal S}^2(\mathbb{R}^l)} + \frac{1}{2}\int_0^T\abs{g^n_u(X^n_u, Y^n_u, Z^n_u)-g_u(X^n_u, Y^n_u, Z^n_u)}^2\\
	\nonumber	&\qquad \qquad + \abs{g^m_u(X^m_u, Y^m_u, Z^m_u)-g_u(X^m_u, Y^m_u, Z^m_u)}^2 du\\
	\label{eq:estimY}	&\quad\quad   + \frac{1}{2}Tk_4^2\rho^2(M)\norm{Y^n-Y^m}_{{\cal S}^2(\mathbb{R}^l)}^2 + \frac{1}{2}\norm{Z^n-Z^m}_{{\cal H}^2(\mathbb{R}^{l\times d})}^2.
	 \end{align*}
	On the other hand, taking conditional expectation in \eqref{eq:estimCauchy} and then the supremum with respect to $t$ and then expectation on both sides, we have due to Young's inequality
	\begin{align}
	\nonumber	&\norm{Y^n- Y^m}_{{\cal S}^2(\mathbb{R}^l)}^2\\
	\nonumber	&\quad \le  E[\abs{h^n(X^n_T) - h(X^n_T)}^2] +  E[\abs{h^m(X^m_T) - h(X^m_T)}^2] + k^2_5\norm{X^n - X^m}_{{\cal S}^2(\mathbb{R}^m)}^2 \\
	\nonumber	&\quad\quad + \frac{1}{2}T\norm{Y^n-Y^m}^2_{{\cal S}^2(\mathbb{R}^l)} + \frac{1}{2}\int_0^T\abs{g^n_u(X^n_u, Y^n_u, Z^n_u)-g_u(X^n_u, Y^n_u, Z^n_u)}^2\\
	\nonumber	&\qquad \qquad + \abs{g^m_u(X^m_u, Y^m_u, Z^m_u)-g_u(X^m_u, Y^m_u, Z^m_u)}^2 du\\
	\nonumber	&\quad\quad \frac{1}{2}Tk_3^2\norm{X^n-X^m}_{{\cal S}^2(\mathbb{R}^m)}^2 + \frac{1}{2}Tk_4^2\rho^2(M)\norm{Y^n-Y^m}_{{\cal S}^2(\mathbb{R}^l)}^2 + \frac{1}{2}\norm{Z^n-Z^m}_{{\cal H}^2(\mathbb{R}^{l\times d})}^2.
	\end{align}
	Combining \eqref{eq:estimX} and \eqref{eq:estimY} we observe that if $T$ is small enough so that
	\begin{equation*}
		\begin{cases}
		3k_1^2T^2\le \frac{1}{2}\\
		\frac{1}{2}T + 3k^2_5k^2_2T^2 + \frac{3}{2}T^3k^2_3k_2^2 + \frac{1}{2}Tk^2_4\rho^2(M)\le \frac{1}{2}
		\end{cases}
	\end{equation*}
	then, the uniform convergence of $(b^n)$, $g^n$ and $(h^n)$ to $b$, $g$ and $h$ ensure that $(X^n,Y^n, Z^n)$ is a Cauchy sequence.
	The verification that the limit $(X,Y,Z)$ of the sequence $(X^n, Y^n,Z^n)$ solves the FBSDE \eqref{eq:fbsde1} uses continuity of the functions $b$, $h$ and $g$, and that $\abs{Z_t}\le M$ is a consequence of the boundedness of $(Z^n)$.
	Taking $C_{k,\lambda,l,d} := \tilde{C}_{k,\lambda,l,d}\wedge \bar{C}_{k,\lambda,l,d} $ concludes the proof.
\hfill $\Box$\\

Due to Theorem \ref{thm:markov1} above, our global existence result now follows from a pasting procedure. 

\subsection{Proof of Theorem \ref{thm:fbsde-global}}

If $T\le C_{k,\lambda,l,d}$, then the result follows from Theorem \ref{thm:markov1}.

Assume $T> C_{k,\lambda,l,d}$ and let $\tilde{h}_M:\mathbb{R}\to \mathbb{R}$ be a continuously differentiable function whose derivative is bounded by $1$ and such that $\tilde{h}_M'(a) = 1$ for all $-M\le a\le M$ and
	\begin{equation*}
	 	\tilde{h}_M(a) = \begin{cases}
	 					(M + 1)& \text{if} \quad\ a > M + 2\\
	 					a &\text{if}  \quad \abs{a} \le M\\
	 					-(M+1) &\text{if} \quad a< -(M+2).
	 	\end{cases}
	 \end{equation*}
	 An example of such a function is given by
	 \begin{equation*}
	 	\tilde{h}_M(a) =\begin{cases}
	 				\left(-M^2 + 2Ma - a(a-4) \right)/4 &\text{if} \quad a \in [M, M+2]\\
	 				\left( M^2 + 2Ma + a(a+4) \right)/4 &\text{if} \quad [-(M+2),-M],
	 	\end{cases}
	 \end{equation*}
	see \citet{Imk-Reis}.
	By the assumptions \ref{a3} the function $\tilde{g}:[0,T] \times\mathbb{R}^m\times\mathbb{R}^{l}\times\mathbb{R}^{l\times d} \to \mathbb{R} $ defined by
	\begin{equation}
	\label{eq:g-tilde2}
		\tilde{g}_t(x,y, z) := g_t(x, y, h_M(z))
	\end{equation}
	with $h_M(z) := (\tilde{h}_M(z^{ij}))_{ij} $ is Lipschitz continuous in all variables.
Thus, it follows from  \citet[Theorem 2.6]{Delarue} that the equation
	\begin{equation}\label{eq:fbsde1_l}
	\begin{cases}
		\tilde{X}_t = x + \int_0^tb_u(\tilde{X}_u, \tilde{Y}_u)\,du + \int_0^t\sigma_u\,dW_u\\
		\tilde{Y}_t = h(\tilde{X}_{T}) + \int_t^T\tilde{g}_u(\tilde{X}_u, \tilde{Y}_u,\tilde{Z}_u)\,du - \int_t^T\tilde{Z}_u\,dW_u, \quad t \in [0,T]
	\end{cases}
	\end{equation}
 	admits a unique solution $(\tilde{X}, \tilde{Y}, \tilde{Z}) \in {\cal S}^2(\mathbb{R}^m) \times {\cal S}^\infty(\mathbb{R}^{l})\times{\cal S}^\infty(\mathbb{R}^{l\times d}) $.
 	Moreover, there exists a Lipschitz continuous function $\theta:[0,T]\times\mathbb{R}^m\rightarrow\mathbb{R}^{l}$ bounded by a constant $K$ such that $\tilde{Y}_t = \theta(t, \tilde{X}_t)$ for all  $t\in[0,T]$.
 	In fact, for every $x,x' \in \mathbb{R}^d$, $t \in [0,T]$ and $i = 1, \dots, l$ we have
 	\begin{align*}
 		&\theta(t,\tilde{X}_t^x) - \theta(t,\tilde{X}^{x'}_t)\\
 		 &= h^i(\tilde{X}_T^x) - h^i(\tilde{X}^{x'}_T) + \int_t^Tg^i_u(\tilde{X}^x_u, \tilde{Y}^x_u, \tilde{Z}^x_u)-g^i_u(\tilde{X}^{x'}_u, \tilde{Y}^{x'}_u, \tilde{Z}^{x'}_u)\,du - \int_t^T\tilde{Z}^{x,i}_u - \tilde{Z}^{x',i}_u\,dW_u\\
 		                       &= h^i(\tilde{X}_T^x) - h^i(\tilde{X}^{x'}_T) + \int_t^T\frac{g^i_u(\tilde{X}^x_u, \tilde{Y}^x_u, \tilde{Z}^x_u)-g^i_u(\tilde{X}^{x'}_u, \tilde{Y}^{x'}_u, \tilde{Z}^{x'}_u)}{\tilde{Z}^{x,i}_u-\tilde{Z}^{x'i}_u}(\tilde{Z}^{x,i}_u-\tilde{Z}^{x'i}_u)1_{\{|\tilde{Z}^{x,i}_u-\tilde{Z}^{x',i}| \neq 0\}}\,du
 \end{align*}
 	\begin{align*}		
 		                       &\quad + \int_t^T(g^i_u(\tilde{X}^x_u, \tilde{Y}^x_u, \tilde{Z}^x)-g^i_u(\tilde{X}^{x'}_u, \tilde{Y}^{x'}_u, \tilde{Z}^{x'}_u))1_{\{|\tilde{Z}^{x,i}_u-\tilde{Z}^{x'i}_u| =0\}}\,du - \int_t^T\tilde{Z}^{x,i}_u - \tilde{Z}^{x',i}_u\,dW_u.
 	\end{align*}
 	Thus, Girsanov's theorem yields
 	\begin{align*}
 		&\abs{\theta^i(t,\tilde{X}^x_t) - \theta^i(t,\tilde{X}^{x'}_t)} \le\\
 		 &E^{Q^{i}}\left[ \abs{h^i(\tilde{X}_T^x) - h^i(\tilde{X}^{x'}_T)} +\int_t^T\abs{g^i_u(\tilde{X}^x_u, \tilde{Y}^x_u, \tilde{Z}^x_u)-g^i_u(\tilde{X}^{x'}_u, \tilde{Y}^{x'}_u, \tilde{Z}^{x'}_u)}1_{\{\abs{\tilde{Z}^{x,i}_u-\tilde{Z}^{x'i}_u}=0\}}\,du\mid {\cal F}_t\right]\\
 									 &\le E^{Q^{i}}\left[ k_5|\tilde{X}_T^x - \tilde{X}^{x'}_T|  +\int_t^T\left(k_3|\tilde{X}^x_u- \tilde{X}^{x'}_u|+ k_4|\tilde{Y}^x_u - \tilde{Y}^{x'}_u|\right)\,du\mid {\cal F}_t\right]
 	\end{align*}
 	where $Q^{i}$ is the probability measure given by
 	\begin{equation*}
 	\frac{dQ^{i}}{dP} = {\cal E}\left( \frac{g^i_u(\tilde{X}^x_u, \tilde{Y}^x_u, \tilde{Z}^x_u)-g^i_u(\tilde{X}^{x'}_u, \tilde{Y}^{x'}_u, \tilde{Z}^{x'}_u)}{\tilde{Z}^{x,i}_u-\tilde{Z}^{x',i}_u}1_{\{|\tilde{Z}^{x,i}_u-\tilde{Z}^{x',i}_u|\neq 0\}}\cdot W \right)_T.
 	\end{equation*}
 	By \ref{a4prime} and boundedness of $\tilde{Y}^{x}$ and $\tilde{Y}^{x'}$ is well defined.
 	Since by Gronwall's lemma we have
 	\begin{equation*}
 		|\tilde{X}^{x}_s - \tilde{X}^{x'}_s| \le (|\tilde{X}^x_t-\tilde{X}^{x'}_t| + k_2\int_s^T|\tilde{Y}^x_u - \tilde{Y}^{x'}_u|\,du)e^{k_1T}, \quad s\in [t,T],
 	\end{equation*}
 	it holds
 	\begin{align*}
 		&\abs{\theta^i(t,\tilde{X}^x_t) - \theta^i(t,\tilde{X}^{x'}_t)} \\
 		&\le E^{Q^{i}}\left[ e^{k_1T}(k_5+Tk_3) |\tilde{X}^x_t-\tilde{X}^{x'}_t| + (k_2k_3Te^{k_1T}+k_4+k_2k_5e^{k_1T})\int_t^T|\theta(u,\tilde{X}^{x}_u) - \theta(u,\tilde{X}^{x'}_u)|\,du \mid {\cal F}_t\right].
 	\end{align*}
 	Hence, $\abs{\theta^i(t,\tilde{X}^x_t) - \theta^i(t,\tilde{X}^{x'}_t)} \le u_t$ where $u_t$ is the solution of the ODE
 	\begin{equation*}
 		u_t = e^{k_1T}(k_5+Tk_3)l|\tilde{X}^x_t-\tilde{X}^{x'}_t|  + (k_2k_3Te^{k_1T}+k_4+k_2k_5e^{k_1T})\int_t^Tlu_s\,du
 	\end{equation*}
 	which is given by
 	\begin{equation*}
 		u_t = e^{k_1T}(k_5+Tk_3)l |\tilde{X}^x_t-\tilde{X}^{x'}_t| \exp\left((k_2k_3Te^{k_1T}+k_4+k_2k_5e^{k_1T})l(T-t)\right).
 	\end{equation*}
 	Thus,
 	\begin{equation*}
 		\abs{\theta(t,\tilde{X}^x_t) - \theta(t,\tilde{X}^{x'}_t)} \le  K_5|\tilde{X}^x_t-\tilde{X}^{x'}_t|,
 	\end{equation*}
 	with
 	$K_5:= \sqrt{l}e^{k_1T}(k_5+Tk_3)l \exp\left((k_2k_3Te^{k_1T}+k_4+k_2k_5e^{k_1T})lT\right)$
 	which show that $\theta$ is a Lipschitz function and the Lipschitz coefficient does not depend on the bound $M$ of $Z$.

	Let $\bar{C}_{k, \lambda,l,d}$ be the constant $C_{k, \lambda,l,d}$ with $k_5$ replaced by $K_5$ and put $N = \lfloor T/\bar{C}_{k, \lambda,l,d}\rfloor$, where $\lfloor a\rfloor$ denotes the integer part of $a$, and $t_i := i\bar{C}_{k,\lambda,l,d}$, $i = 0, \dots, N$ and $t_{N+1} = T$.
 Since $t_1\le \bar{C}_{k,\lambda,l,d}$, by Theorem \ref{thm:markov1} the FBSDE
	\begin{equation*}
	\begin{cases}
		X_t = x + \int_0^tb_u(X_u, Y_u )\,du + \int_0^t\sigma_u\,dW_u\\
		Y_t = \theta(t_1,X_{t_1}) + \int_t^{t_1}g_u(X_u, Y_u,Z_u)\,du - \int_t^{t_1}Z_u\,dW_u, \quad t \in [0,t_1]
	\end{cases}
\end{equation*}
	admits a unique solution $(X^1, Y^1, Z^1)$ such that $\abs{Z_t^1} \le \bar{M}$ with $\bar{M} = 4\lambda_2K_5\sqrt{dl}$ for all $t \in [0,t_1]$.
	Therefore, $(X^1, Y^1, Z^1)1_{[0,t_1]} = (\tilde{X}, \tilde{Y}, \tilde{Z})1_{[0,t_1]}$.
	Similarly, we obtain a family $(X^i, Y^i, Z^i)$ of solutions of the FBSDEs
	\begin{equation*}
	\begin{cases}
		X_t = \tilde{X}_{t_{i-1}} + \int_{t_{i-1}}^tb_u(X_u, Y_u )\,du + \int_{t_{i-1}}^t\sigma_u\,dW_u\\
		Y_t = \theta(t_i,X_{t_i}) + \int_t^{t_{i}}g_u(X_u, Y_u,Z_u)\,du - \int_t^{t_{i}}Z_u\,dW_u, \quad t \in [t_{i-1},t_{i}]
	\end{cases}
\end{equation*}
	such that $(X^i, Y^i, Z^i)1_{[t_{i-1},t_i]} = (\tilde{X}, \tilde{Y}, \tilde{Z})1_{[t_{i-1},t_{i}]}$, $i = 1,\dots,N+1$.
	Define
	\begin{equation*}
		X := \sum_{i=1}^{N+1}X^i1_{[t_{i-1}, t_i]}; \quad Y := \sum_{i=1}^{N+1}Y^i1_{[t_{i-1}, t_i]} \quad \text{and } Z := \sum_{i=1}^{N+1}Z^i1_{[t_{i-1}, t_i]}.
	\end{equation*}
	Then, $(X, Y, Z) \in  {\cal S}^2(\mathbb{R}^m) \times {\cal S}^\infty(\mathbb{R}^{l}) \times {\cal S}^\infty(\mathbb{R}^{l\times d})$ is the unique solution of the FBSDE \eqref{eq:fbsde1} satisfying $\abs{Z_t} \le \bar{M}$ for all $t \in [0, T]$.
		In fact, it is clear that $(X, Y, Z) \in  {\cal S}^2(\mathbb{R}^m) \times {\cal S}^\infty(\mathbb{R}^{l}) \times {\cal S}^\infty(\mathbb{R}^{l\times d})$ as a finite sum of elements of the same space.
	Let $t \in [0,T]$ and $i = 1, \dots, N+1$ such that $t \in [t_{i-1}, t_i]$.
	We have
	\begin{align*}
		x + \int_0^tb_u(X_u)\,du  + \int_0^t\sigma_u\,du
										    &= x + \sum_{j=1}^i\left( \int_{t_{j-1}}^{t_j\wedge t} b_u(X^j_u)\,du + 	\int_{t_{j-1}}^{t_j\wedge t}\sigma_u\,dW_u \right)\\
										   & = X^i_t =  X_t
	\end{align*}	
	 and
	 \begin{align*}
	 	h(X_T)& + \int_t^Tg_u(X_u, Y_u, Z_u)\,du - \int_t^TZ_u\,dW_u\\
	 	      & = h(X^{N+1}_T) + \sum_{j = i}^{N+1} \left( \int_{t_{j-1}\vee t}^{t_j}g_u(X^j_u, Y^j_u, Z^j_u)\,du - \int_{t_{j-1}\vee t}^{t_j}Z^j_u\,dW_u \right)
	 	       = Y^i_t = Y_t.
	 \end{align*}
	 That is, $(X, Y, Z)$ satisfies Equation \eqref{eq:fbsde1}.
\hfill $\square$

\subsection{Proof of Proposition \ref{thm:prop_coupled} }

\label{sec:FBSDE1}
By Theorem \ref{thm:fbsde-global} the FBSDE
\begin{equation*}
	\begin{cases}
		\tilde{X}_t &=x+\int_{0}^{t}\tilde{b}_s(\tilde{X}_s,\tilde{Y}_s)\,ds+\int_{0}^{t}\sigma_sdW_s\\
	\tilde{Y}_t &=\tilde{h}(\tilde{X}_T)+\int_{t}^{T}\tilde{g}_s(\tilde{X}_s,\tilde{Y}_s,\tilde{Z}_s)\,ds-\int_{t}^{T}\tilde{Z}_sdW_s, \quad t\in [0,T]
	\end{cases}
\end{equation*}

has a unique global solution $(\tilde{X}, \tilde{Y}, \tilde{Z}) \in {\cal S}^2(\mathbb{R}^m)\times {\cal S}^2(\mathbb{R}^{l})\times {\cal S}^\infty(\mathbb{R}^{l\times d})$ such that $|\tilde{Z}_t|\leq \tilde{M}$ for some constant $\tilde{M} \ge 0$. Let $X_t=\tilde{X}_t, Y_t=\Gamma^{-1}\tilde{Y},Z_t=\Gamma^{-1}\tilde{Z}$, we obtain that $(X, Y, Z) \in {\cal S}^2(\mathbb{R}^m)\times {\cal S}^2(\mathbb{R}^{l})\times {\cal S}^\infty(\mathbb{R}^{l\times d})$ such that $|Z_t|\leq \bar{M}$ for some constant $\bar{M} \ge 0$. Moreover, $(X,Y,Z)$ satisfies the FBSDE \eqref{eq:fbsde1}. This completes the proof.

\hfill $\square$

\bibliographystyle{abbrvnat}
\bibliography{reference-qfbsde}

\end{document}